\documentclass[12pt,reqno]{article}

\usepackage[usenames]{color}
\usepackage{subcaption}
\usepackage{amssymb}
\usepackage{amsmath}
\usepackage{amsthm}
\usepackage{amsfonts}
\usepackage{amscd}
\usepackage{graphicx}
\usepackage{tikz}
\usetikzlibrary{automata, positioning}

\usepackage[colorlinks=true,
linkcolor=webgreen,
filecolor=webbrown,
citecolor=webgreen]{hyperref}

\definecolor{webgreen}{rgb}{0,.5,0}
\definecolor{webbrown}{rgb}{.6,0,0}

\usepackage{color}
\usepackage{fullpage}
\usepackage{float}
\usepackage{hyperref}

\usepackage{graphics}
\usepackage{latexsym}
\usepackage{epsf}
\usepackage{breakurl}

\usepackage{hyperref,xcolor}
\usepackage[T1]{fontenc}
\usepackage{pdfsync}
\usepackage{yfonts}      
\usepackage{tcolorbox}   
\usepackage{lipsum}

\newcommand{\seqnum}[1]{\href{https://oeis.org/#1}{\rm \underline{#1}}}

\begin{document}


\theoremstyle{plain}
\newtheorem{theorem}{Theorem}
\newtheorem{corollary}[theorem]{Corollary}
\newtheorem{lemma}[theorem]{Lemma}
\newtheorem{proposition}[theorem]{Proposition}
\newtheorem{remarks}[theorem]{Remarks}

\theoremstyle{definition}
\newtheorem{definition}[theorem]{Definition}
\newtheorem{example}[theorem]{Example}
\newtheorem{conjecture}[theorem]{Conjecture}

\theoremstyle{remark}
\newtheorem{remark}[theorem]{Remark}

        \newcommand{\Gandhar}[1]{\marginpar{\small \raggedright \textcolor{orange}{{\bfseries Gandhar:} #1}}}
        \newcommand{\gandhar}[1]{\textcolor{orange}{ #1}}
\begin{center}
\vskip 1cm{\LARGE\bf Anti-Recurrence Sequences
}
\vskip 1cm
\large
Robbert Fokkink\\
Technical University Delft\\ Faculty of Mathematics\\ Mekelweg 4, 2628CD\\
The Netherlands\\
\href{mailto:email}{\tt r.j.fokkink@tudelft.nl} \\
\vskip 0.2cm
Gandhar Joshi\\
School of Mathematics and Statistics\\ 
The Open University\\
Walton Hall\\
Milton Keynes, MK7 6AA\\ 
UK\\
\href{mailto:email}{\tt gandhar.joshi@open.ac.uk}
\end{center}

\vskip .2 in
\begin{abstract}
We extend previous work on anti-recurrence sequences of Kimberling and Moses, Zaslavsky, and Bosma et al. 
Kimberling and Moses have formulated several questions on these sequences, which can be combined into 
the meta-conjecture that anti-recurrence sequences are sums of linear progressions and automatic sequences. 
We solve this conjecture under a restriction on the linear form that generates the anti-recurrence.
\end{abstract}

In a linear recurrence sequence, each term is a linear combination of the ones that came before it. The study
of these sequences is a topic in itself~\cite{everest}. The first example that jumps to 
everyone's mind is the Fibonacci sequence
\[
F_{n+1}=F_n+F_{n-1},
\]
starting from the initial conditions $F_0=0$ and $F_1=1$. These numbers have a generative nature. 
Starting from the immortal Eve, $F_n$ is the number
of female rabbits in the $n$-th generation, if each mature female produces the perfect offspring of one female and one male baby rabbit, 
{which mature in the next generation, ready to reproduce. This happens in perfect sync, forever and ever; as}
a mathematical abstraction, of course. 
In reality, the average number of babies per litter is 4.5, and female rabbits produce up to ten litters
per lifetime~\cite{rabbit}.

Recurrence sequences are defined by earlier terms in the sequence. In contrast with this, the \emph{anti-recurrence}
sequences, which we consider in this paper, are defined by earlier terms that are \emph{not} in the sequence. 
The anti-Fibonacci numbers start with $A_0=0$. They extend by the rule 
\newline
\vspace{-0.5cm}
\begin{center}
\fbox{
  \parbox{0.7\linewidth}{
    \centering
    {The next anti-Fibonacci number is the sum of the two\\
    most recent NON-members of the anti-Fibonacci sequence.}
  }
}
\end{center}
The first two non-members 1 and 2 add up to the anti-Fibonacci $A_1=3$. 
The next two non-members are 4 and 5, which add
up to the anti-Fibonacci $A_2=9$, etc.
This sequence is listed under \seqnum{A075326} in the On-Line Encyclopedia of Integer Sequences.
\[
0, 3, 9, 13, 18, 23, 29, 33, 39, 43, 49, 53, 58, 63, 69, 73, 78, 83, 89, 93, 98, 103, 109, 113,\ldots
\]
It was entered into the OEIS by Amarnath Murthy and was named anti-Fibonacci by Douglas Hofstadter in an unpublished note~\cite{hofstadter}. He observed that the first difference sequence
\begin{equation}\label{eq:H}
3, 6, 4, 5, 5, 6, 4, 6, 4, 6, 4, 5, 5, 6, 4, 5, 5, 6, 4, 5, 5, 6, 4,\ldots
\end{equation}
consists of the two-letter words $64$ and $55$, apart from the initial letter $3$. 
This is \seqnum{A249032} in the OEIS.
All numbers with
final digit $3$ are anti-Fibonaccis, and the other anti-Fibonaccis either end with a 9 or an 8. 
Hofstadter observed, without giving a proof, that the pattern of 9's and 8's can be generated from a period-doubling substitution
\[
9\mapsto 98,\ 8\mapsto 99.
\]
The proof was supplied by Thomas Zaslavsky, in another unpublished note~\cite{zaslavsky}. In particular,
he gave an explicit equation for the anti-Fibonacci numbers
\begin{equation}\label{eq:0}
 \text{For all }n\geq1,\quad\seqnum{A075326}(n)-5n+2=\texttt{PD}_{n-1}.
\end{equation}
The period doubling sequence $\mathtt{PD}_n$ consists of zeros and ones and is generated by
\[
0\mapsto 01,\ 1\mapsto 00,
\]
starting from $\mathtt{PD}_0=0$. It is entry \seqnum{A096268} in the OEIS. Note that the 
indexing runs from~0 and not from~1. This is culturally defined. 
\smallbreak
\textbf{CAUTION:} For automatic sequences such as \texttt{PD}, indexing starts at zero.
\smallbreak
\smallbreak
Clark Kimberling and Peter Moses studied the more general
class of complementary sequences~\cite{kimberlingmoses},
for which anti-recurrence sequences are a special case.
They observed some properties of anti-recurrence sequences, which
Kimberling entered as conjectures under
\seqnum{A265389}, \seqnum{A299409}, \seqnum{A304499}, and \seqnum{A304502} in the OEIS. The conjectures
for the first two sequences were verified by Bosma et al.~\cite{walnutoeis}
using Hamoon Mousavi's automatic theorem prover \texttt{Walnut}~\cite{mousavi}.
We settle the other two conjectures on \seqnum{A304499} and \seqnum{A304502} in this paper, again with the assistance of \texttt{Walnut}.
These conjectures can be combined into a meta-conjecture, which specifies the discussion in section six of \cite{kimberlingmoses}. It was named the Clergyman's conjecture in~\cite{walnutoeis}.
\begin{conjecture}
    Every anti-recurrence sequence is a sum of a linear sequence and an automatic sequence.
\end{conjecture}
Our paper is organized as follows. In section 1 we review the basic notions. Section 2 settles
the conjectures of Kimberling for \seqnum{A304499} and \seqnum{A304502} using \texttt{Walnut}. 
In section 3 we extend
the results of Bosma et al.\hspace{-0.1cm} and solve the conjecture for anti-bonaccis. 
Our main result, Theorem~\ref{main} in section 4, settles the conjecture under a restriction
on the linear form that generates it. We
are unable to settle the full conjecture.

\section{Definitions, notation, preliminary results}

All numbers are natural numbers (positive integers) unless stated otherwise. We will write sequences in capitals. 
Two strictly increasing sequences $A_n$ and $B_n$ of natural numbers 
are \emph{complementary} if
every natural number is either in $A_n$ or in $B_n$ exclusively.
Let $\mathbf a=(a_1,\ldots,a_k)$ be a
positive integral vector of dimension $k>1$ and all $a_i>0$.
Let $f(\mathbf x)=\mathbf a\cdot \mathbf x$ be its associated linear form. We say that $A_n$ is an
\emph{anti-recurrence sequence} if
\begin{equation}\label{eq:1}
A_{n}=f(B_{(n-1)k+1},B_{(n-1)k+2},\ldots,B_{nk})=\sum_{j=1}^k a_j B_{(n-1)k+j}.    
\end{equation}
The \emph{trace} $\tau$ of the linear form is $\tau=\sum_{i=1}^m a_i\geq k$.

In a precise but elaborate naming convention, $A_n$ is the anti-recurrence sequence and its complement $B_n$ is the \emph{non-anti-recurrence sequence}.
We say that a set $\{B_{jk+1},B_{jk+2},\ldots,B_{(j+1)k})$ is the \emph{$B$-block}
that generates $A_{j+1}$. Note that we use $X_n$ both for the sequence and the individual
number and the context will make clear what we mean. 
If numbers $\{a,a+1,\ldots,b\}$ are consecutive, then we say that they
form the {interval} $[a,b]$. 

 \begin{lemma}\label{lem:2}
        Successive anti-recurrence numbers satisfy \[A_{n+1}-A_{n}\geq k\tau.\] 
        The above is an equality if the $B$-blocks for both $A_{n+1}$ and $A_n$
        are intervals such that their union is also an interval. In particular, 
        the inequality is strict if one block is an
        interval and the other is not.
    \end{lemma}
    \begin{proof}
        Let $\mathbf {a}=(a_1,\ldots,a_k)$ and let $B_n$ be the non-anti-recurrence sequence. 
        We have that \[A_n=\sum_{j=1}^k a_jB_{m+j}\] for $m=k(n-1)$ and
        \[A_{n+1}=\sum_{j=1}^k a_jB_{k+m+j}.\]
        Now $B_{k+m+j}-B_{m+j}\geq k$ since this sequence is increasing. If the $B$-blocks are consecutive
        intervals, then this is an equality. If one is an interval and the other is not, then
        there must be a $j$ such that $B_{k+m+j}-B_{m+j}> k$.
    \end{proof}

The \emph{mex} or minimal excluded value of $S\subset \mathbb N$ is
\[
\text{mex}(S)=\text{min}\left(\mathbb N\setminus S\right).
\]
It comes up naturally in anti-recurrences sequences, as
observed by Kimberling and Moses.

\begin{lemma}\label{lem:1}
    A positive linear form $\mathbf a$ determines $A_n$ and $B_n$. 
\end{lemma}
\begin{proof}
    We need to show that the complementary sequences $A_n$ and $B_n$ exist and
    are unique. 
    Let $\mathcal A_n=\{A_i\colon i\leq n\}$ be the initial anti-recurrences and let $\mathcal B_{kn}=\{B_j\colon j\leq kn\}$ be the initial $B$-blocks.
    By induction, we assume that \[[1,B_{kn}]\subset\mathcal A_n\cup\mathcal B_{kn}.\]
    Let \[b=\text{mex}\left(\mathcal A_n\cup \mathcal B_{kn}\right).\] The interval $[b,b+k-1]$
    has length $k$ and by Lemma~\ref{lem:2} can contain at most one number from $\mathcal A_n$.
    If it contains no such number, then $[b,b+k-1]$ must be the $B$-block from
    $B_{kn+1}$ up to $B_{k(n+1)}$ since the sequence $B_n$ consists of all numbers that are not in $A_n$. 
    If one of the numbers in $[b,b+k-1]$ is in $\mathcal A_n$, then the $B$-block from
    $B_{kn+1}$ up to $B_{k(n+1)}$
    skip that number. In any case, the next $B$-block from
    $B_{kn+1}$ up to $B_{k(n+1)}$ is uniquely determined by a mex-rule and generates $A_{n+1}$.
\end{proof}

The linear form $\mathbf a=(1,1)$ gives the anti-Fibonaccis. Its complementary sequence
\seqnum{A249031} 
\[
1,2, 4, 5, 6, 7, 8, 10, 11, 12, 14, 15, 16, 17, 19, 20, 21, 22,\ldots
\]
is the \emph{non-anti-Fibonacci sequence}. 
The proof of Lemma~\ref{lem:1} shows that the $B_n$ are defined blockwise by the mex. It
is convenient to cut up these blocks into individual parts and define the $k$ subsequences
\[B^j_n=B_{j+(n-1)k}\] 
for $j=1,\ldots,m$. For instance,
the non-anti Fibonacci sequence is divided into \seqnum{A075325} 
\[
1, 4, 6, 8, 11, 14, 16, 19, 21,\ldots
\]
and \seqnum{A047215}
\[
2, 5, 7, 10, 12, 15, 17, 20, 22,\ldots
\]
We can generate the sequences $B_n^j$ and $A_n$ simultaneously, adding the mex to each
sequence $B_n^j$ from $j=1$ to $j=m$, and then $A_n=\sum_j a_jB_n^j$. This is how
Kimberling and Moses define anti-recurrence sequences.

\medbreak
A \text{Deterministic Finite State Automaton with Output}, or DFAO, is the simplest type of
computing machine. It is able to read a finite input word and returns an output. A DFAO is a 6-tuple
\( \mathcal{A} = (Q, \Sigma, \delta, q_0, \Gamma, \lambda) \), where:
\begin{itemize}
    \item \( Q \) is a finite set of \textbf{states},
    \item \( \Sigma \) is a finite \textbf{input alphabet},
    \item \( \delta: Q \times \Sigma \to Q \) is the \textbf{transition function},
    \item \( q_0 \in Q \) is the \textbf{initial state},
    \item \( \Gamma \) is a finite \textbf{output alphabet},
    \item \( \lambda: Q \times \Sigma \to \Gamma \) is the \textbf{output function}.
\end{itemize}
For instance, there is a DFAO for the Period Doubling sequence that returns the digit ${PD}_n$ 
upon input $n$ in binary{;} see Fig.~\ref{fig:1}. According to Cobham's little theorem~\cite{shallit2022},
a DFAO corresponds to a substitution $\sigma$. To a state $a$, it assigns the word $\sigma(a)$ such
that the $j$-the digit of $\sigma(a)$ corresponds to the transition from $a$ under $j$. The DFAO
in Fig.~\ref{fig:1} 
corresponds to the Period Doubling substitution $a\mapsto ab,\ b\mapsto aa$.
\begin{figure}
    \centering
    \includegraphics[width=0.5\linewidth]{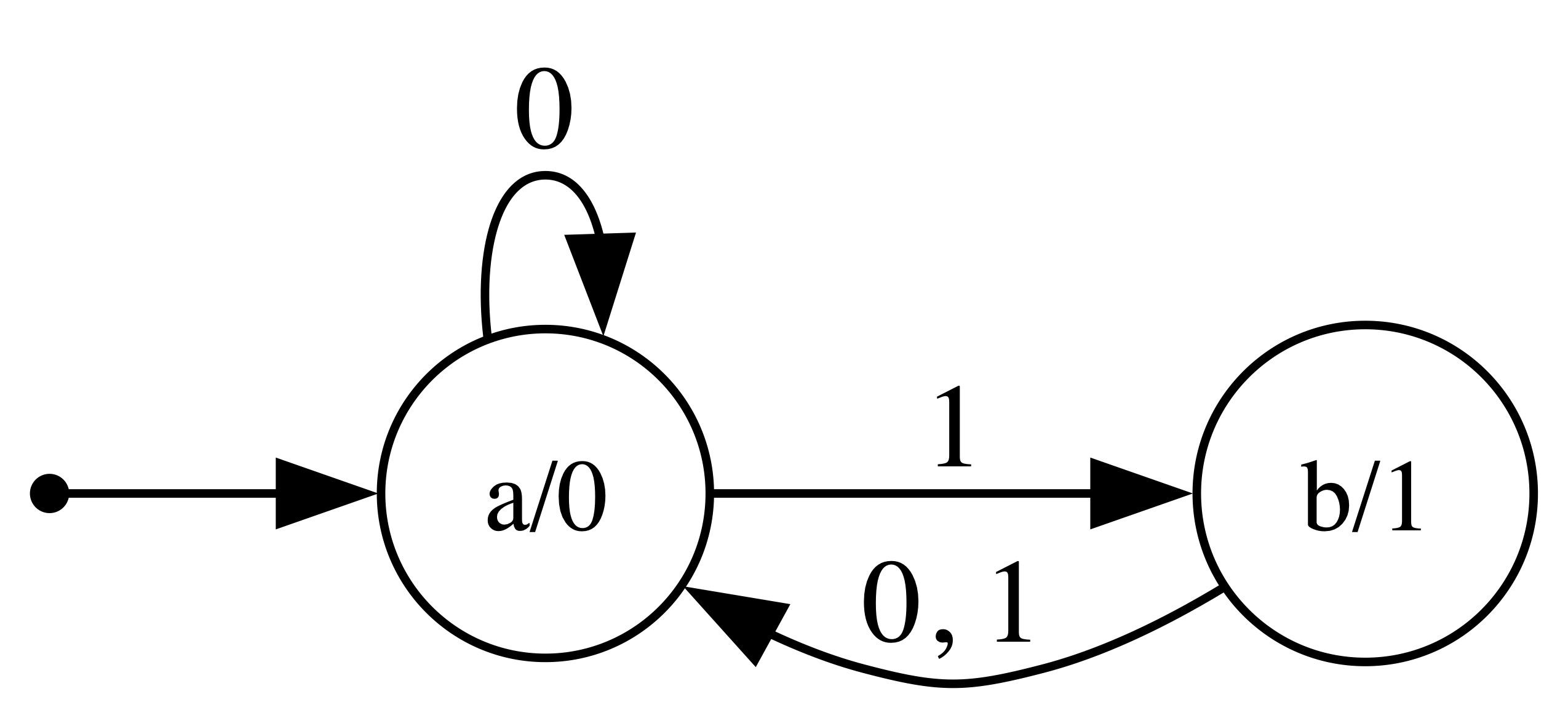}
    \caption{The automaton for the Period Doubling word $0100010101010001\cdots$. 
    Input of  $\mathtt{PD}_n$ is in binary and
    starts at $n=0$ instead of $n=1$.
    The state $a$ outputs $0$ and state $b$ outputs $1$.
    For instance, $n=9$ is expanded as $1001$ in binary and has digit $\mathtt{PD}_9=1$.}
    \label{fig:1}
\end{figure}
A {\textit{$k$-DFAO}} is an automaton with alphabet $\Sigma=\{0,1,\ldots,k-1\}$.
It reads numbers that are expanded in base $k$. Our DFAO
for the Period Doubling word is a $2$-DFAO. A sequence $X_n$ is $k$-automatic
if there exists a $k$-DFAO that gives output $X_n$ on input $n$. Similar case where $X_n$ is both a term and a sequence in one sentence.
\medbreak
We will use the automatic theorem prover \texttt{Walnut}.
It has a transparent syntax that can be easily understood, even by readers that are unfamiliar with the
software. We refer to Hamoon Mousavi's user manual~\cite{mousavi} and Jeffrey Shallit's textbook~\cite{shallit2022}
for more information.
In the words of Jeffrey Shallit~\cite{sumsets}, \texttt{Walnut} serves as a telescope to
view results that at first appear only distantly provable, and that is how we use it in this paper.
\medbreak
We now give a more precise statement of 
Conjecture \ref{conj:1}.
\begin{conjecture}\label{conj:1}
    Let $\mathbf a=(a_1,\ldots,a_k)$ be positive and integral
    of dimension $k>1$.
    Let $A_n$ be the anti-recurrence sequence for the
    linear form $f(\mathbf x)=\mathbf a\cdot\mathbf x$.
    Then $A_n-\kappa n$ is $\tau$-automatic for $\kappa=k\tau+1$
    and $\tau$ the trace of the linear form. 
\end{conjecture}

We prove this conjecture under a restriction on $\mathbf{a}$ in the final
section of our paper.
Thomas Zaslavsky proved it for $\mathbf{a}=(1,1)$ and
Bosma et al. proved it
for $\mathbf a=(1,1,1)$ and $\mathbf a=(1,1,1,1)$,
naming it the Clergyman's Conjecture. 
A weak form of the conjecture says that the difference sequence $A_n-\kappa n$ is
bounded. 
In fact, this is how Kimberling and Moses formulate their conjectures,
but they do provide conjectured substitutions that generate
the specific difference sequences
in~\cite{kimberlingmoses}. We confirm these substitutions in Theorem~\ref{thm:abonacci}

\section{The anti-Pell and anti-Jacobsthal numbers}

The recurrence $X_{n+1}=2X_n+X_{n-1}$ generates the Pell numbers \seqnum{A000129} and $X_{n+1}=X_n+2X_{n-1}$ generates
the Jacobsthal numbers \seqnum{A001045}. We consider their counterparts, the anti-recurrence sequences for 
$\mathbf {a}=(1,2)$ and $\mathbf {a} = (2,1)$. Kimberling conjectured on the OEIS
that the difference sequence is bounded for these anti-recurrences. 

For $\mathbf a=(1,2)$, we get the anti-Pell numbers 
\seqnum{A304502}
\[
5, 11, 20, 26, 34, 41, 47, 53, 61, 68, 74, 83, 89, 95, 103, 110,\ldots.
\]
Here we ignore $A_0=0$. 
Observe that the subsequence $A_{3n+1}=5+21n$ forms an arithmetic progression,
in analogy of what we saw for the anti-Fibonaccis.
The differences
between consecutive numbers now show a period three
\[
6, 9, 6, 8, 7, 6, 6, 8, 7, 6, 9, 6, 6, 8, 7, \ldots
\]
with blocks $696$, $876$, and $687$. This is in line with the meta-conjecture that the
difference sequence must be $3$-automatic. On the OEIS, Kimberling conjectures that
\[
0\leq A_n-7n+3\leq 2.
\]
We apply the method of guessing an automaton as described in
~\cite[p. 75]{shallit2022} to guess a DFAO
for the difference sequence $A_n-7n$. We shift the index by one 
to comply with the convention that automatic sequences start at index~$0$
and we adjust the sequence 
to $A_{n+1}-7n-4$, to make the output alphabet $\Gamma=\{0,1,2\}$, as shown in Fig.~\ref{fig:2}.
We call this automaton \texttt{a12} and we implement the anti-Pell numbers
in \texttt{Walnut} by the command
\vspace{-0.5cm}
\begin{flushleft}
\texttt{
\newline
def a304502 "?msd\_3 (n>0) \& s=(7*n-3+a12[n-1])": 
}
\end{flushleft}
In particular, the variable $s$ is equal to $A_n$. 
\begin{figure}
    \centering
    \includegraphics[width=0.6\linewidth]{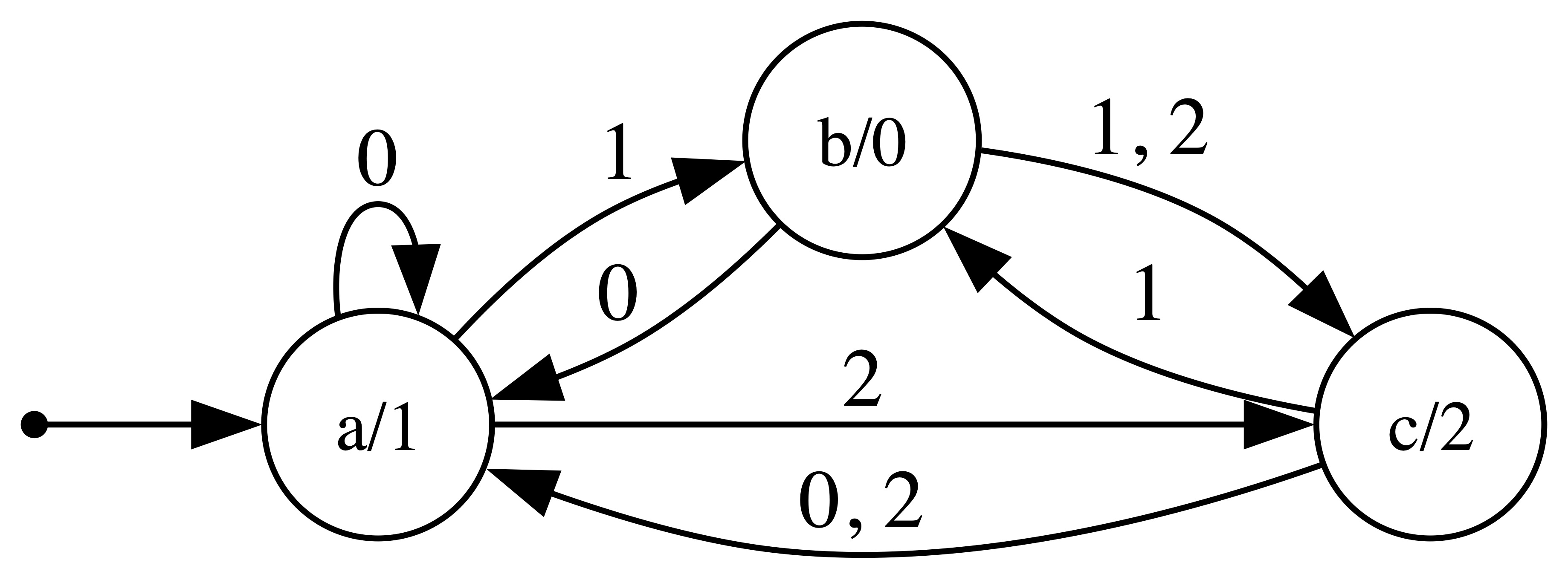}
    \caption{A 3-DFAO with output for the difference sequence $A_{n+1}-7n-4$,
    of the anti-Pell numbers.
    All inputs $n=0\pmod 3$ end in state $a$ with output $1$.
    This means that $A_{n}=7n-2$ if $n=1\pmod 3$. For instance, $A_{12}$ has $n+1=12$
    which is entered as $n=102$ in the expansion with respect to base 3. It ends
    in state $c$. Therefore, $A_{12}=7\cdot 11+6=83.$}
    \label{fig:2}
\end{figure}

To verify Kimberling's conjecture for anti-Pell numbers, we also need the two
non-anti-Pell sequences $B_n^1$ and $B_n^2$. Kimberling's full conjecture is that
\begin{equation}\label{eq:2}
\begin{aligned}
0 &\leq A_n - 7n + 3 &\leq 2, \\
0 &\leq 3B_n^1 - 7n + 6 &\leq 3, \\
0 &\leq 3B_n^2 - 7n + 2 &\leq 3.
\end{aligned}
\end{equation}
Observe that successive entries $B_n^1, B_n^2$ in a non-anti-recurrence sequence
differ by one or two, because there can be at most one anti-recurrence number in between.
Since $2B_n^2+B_n^1=A_n$, it follows that $3B_n^2=A_n+i$ for $i\in\{1,2\}$, and this can
be used to obtain the third inequality from the first. 
In the OEIS, $B_n^1$ is \seqnum{A304500} 
\[
1, 3, 6, 8, 10, 13, 15, 17, 19, 22, 24, 27, 29, 31, 33, 36,\ldots,
\]
and $B_n^2$ is \seqnum{A304501}
\[
2, 4, 7, 9, 12, 14, 16, 18, 21, 23, 25, 28, 30, 32, 35, 37,\ldots.
\]
It is possible to derive these sequences from the anti-Pell numbers.
As already observed, $B_n^2$ is equal to $(A_n+2)/3$ rounded down.
If we have $A_n$ and $B_n^2$ then we also have $B_n^1=A_n-2B_n^2$.
The non-anti-recurrence sequences are implemented by the following commands

\vspace{-0.5cm}
\begin{flushleft}
\texttt{
\newline
def a304501 "?msd\_3 Es \$a304502(n,s) \& t=(s+2)/3":
\newline
def a304500 "?msd\_3 Es,t \$a304502(n,s) \& \$a304501(n,t) \& u+2*t=s":
}
\end{flushleft}
The automata for the difference sequences $B_n^1$ and $B_n^2$ are shown in Fig.~\ref{fig:3}.

\begin{figure}[ht]
    \centering
    \begin{subfigure}{0.45\textwidth}
        \includegraphics[width=0.9\linewidth]{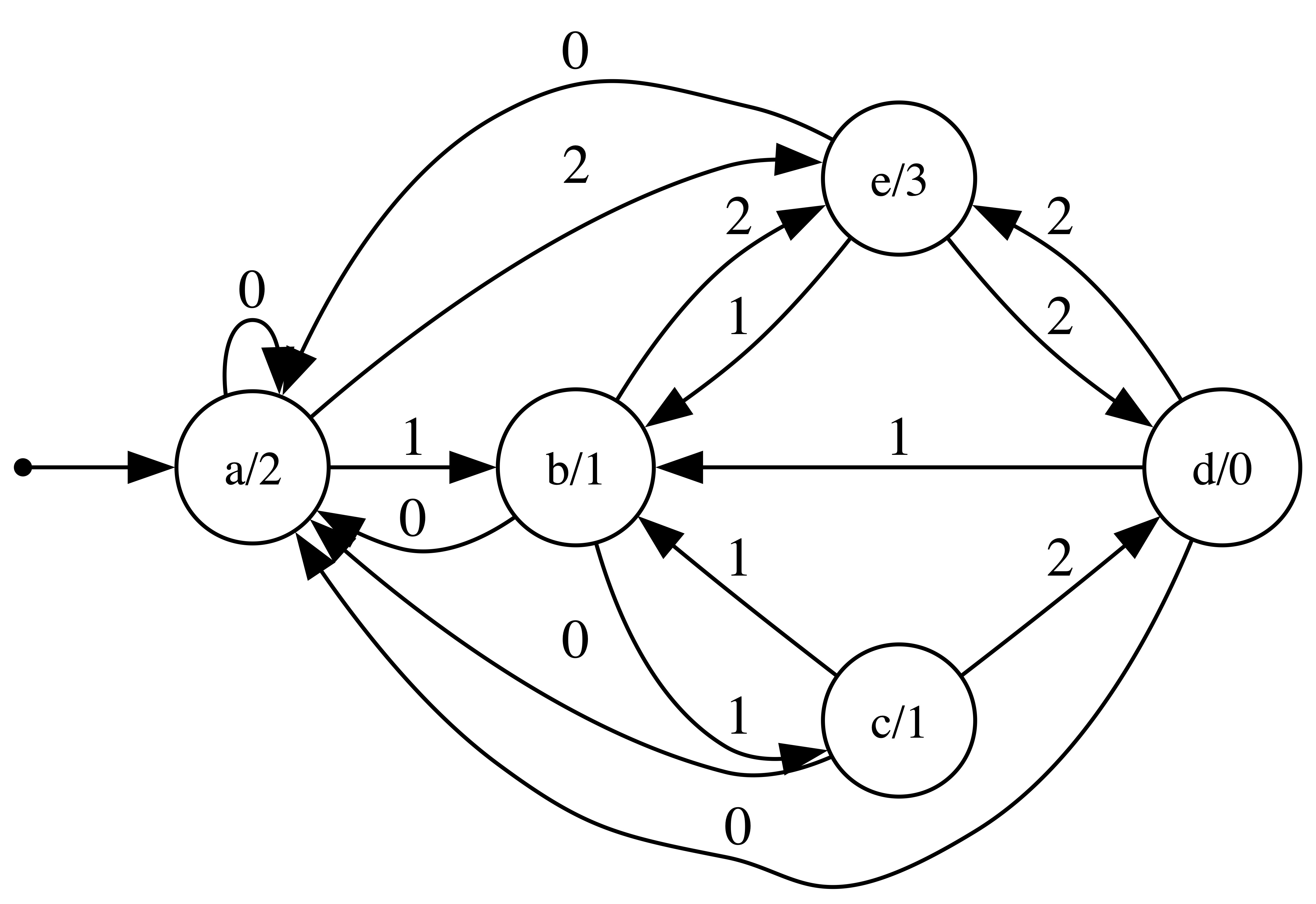}
    \end{subfigure}
    \begin{subfigure}{0.45\textwidth}
       \includegraphics[width=0.9\linewidth]{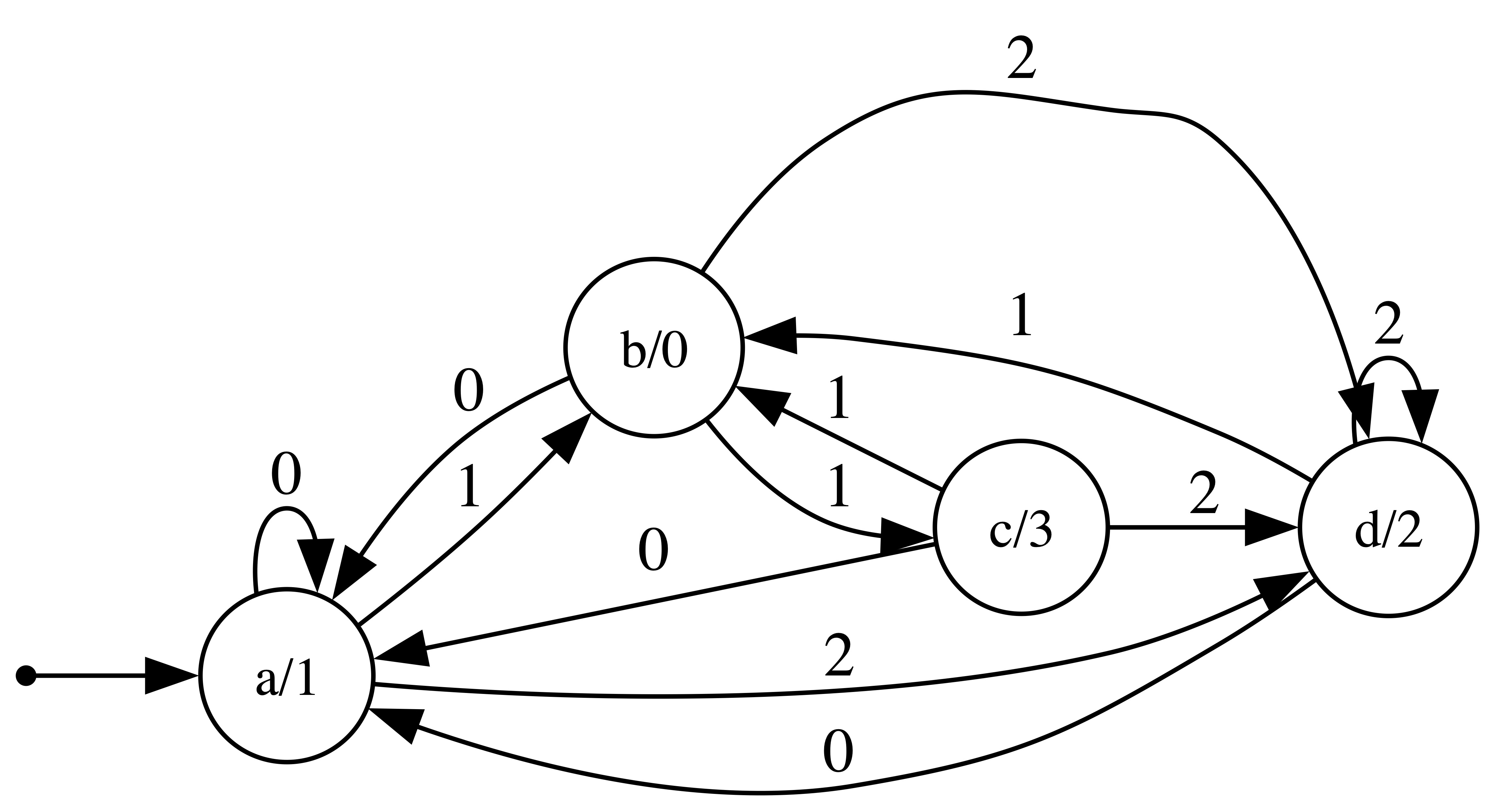} 
    \end{subfigure}
    \caption{The 3-DFAO's for the non-anti-recurrences
    $3B_{n+1}^1-7n-1$ (left) and $3B_{n+1}^2-7n-5$ (right). For instance,
    $B_{8}^2$ can be computed from the input $7$, which is $21$ in
    base 3. It ends in state $b$ with output $0$ and therefore $B_8^2=18$.
    Observe that the final digit determines the output mod 3.
    These DFAO's are from~\cite{walnutoeis}, where they were given in
    lsd form. They were converted to msd by \texttt{Walnut}.
    }
    \label{fig:3}
\end{figure}
\begin{theorem}
    The anti-Pell numbers satisfy Kimberling's bounds in Equation~\ref{eq:2}.
\end{theorem}
\begin{proof}
    The outputs of the DFAO's in Fig.~\ref{fig:2} and Fig.~\ref{fig:3} are within Kimberling's bounds.
    Note that we shifted $A_n-7n+3$ to $A_{n+1}-7(n+1)+3$, and likewise for $B_n^1$ and $B_n^2$, to comply with the rule that
    automatic sequences start at index $0$.
    We need to verify that these DFAO's indeed correspond to the difference sequences for the anti-Pell numbers,
    which we do with the assistance of \texttt{Walnut}.

According to Lemma~\ref{lem:1}, the sequences are determined by
their initial values and a mex rule.
We first check that the initial values are
$B_1^1=1,\ B_1^2=2, A_1=5$.
\vspace{-0.5cm}
\begin{flushleft}
\texttt{
\newline
eval test "?msd\_3 \$a304500(1,1) \& \$a304501(1,2) \& \$a304502(1,5)": 
}
\end{flushleft}
\texttt{Walnut} evaluates the statement as \texttt{TRUE}. 

We check the mex rule for $B_n^1$, which says that it is the least new number after the first $n$ have
been defined. In first-order logic, the statement is
\[
\forall n,s,t\in\mathbb N \ t<s\ \wedge\ s=B_n^1\ \implies\ \exists m<n\ (t=B_m^1\ \vee\ t=B_m^2\ \vee\ t=A_m)  
\]
In \texttt{Walnut} this statement becomes
\vspace{-0.5cm}
\begin{flushleft}
\texttt{
\newline
eval testB1 "?msd\_3 An,s,t (\$a304500(n,s) \& t>0 \& t<s) => (Em (m<n) \& (\$a304500(m,t)|\$a304501(m,t)|\$a304502(m,t)))": 
}
\end{flushleft}
It is evaluated as \texttt{TRUE}. The mex condition on $B_n^2$ is that it is the first new number after $B_n^1$.
\[
\forall n,s\in\mathbb N \ \left(s>1\ \wedge\ s=B_n^2\right)\ \implies\ \left(s-1=B_n^1\ \vee\ \exists \ m<n\ s-1=A_m\right)  \]
In \texttt{Walnut} this statement is
\vspace{-0.5cm}
\begin{flushleft}
\texttt{
\newline
eval testB2 "?msd\_3 An,s (s>1\ \&\ \$a304501(n,s)) => (\$a304500(n,s-1)|(E m (m<n) \& \$a304502(m,s-1)))": 
}
\end{flushleft}
It is evaluated as \texttt{TRUE}. The final condition is that $A_n=B_n^1+B_n^2$.
\vspace{-0.5cm}
\begin{flushleft}
\texttt{
\newline
eval testA "?msd\_3 An,s,t (\$a304500(n,s) \& \$a304501(n,t)) => \$a304502(n,s+2*t)": 
}
\end{flushleft}
It is evaluated as \texttt{TRUE}. This proves that these are indeed the anti-Pell sequence and
its non-anti-Pell counterparts. Conjecture~\ref{conj:1} holds for $\mathbf a=(1,2)$ and Kimberling's
bounds in Equation~\ref{eq:2} apply.
\end{proof}

For $\mathbf a=(2,1)$ we get the anti-Jacobsthal numbers 
\seqnum{A304499}
\[
4, 11, 19, 25, 32, 40, 46, 52, 61, 67, 74, 82, 88, 95, 103, 109,\ldots
\]
which resemble the anti-Pell numbers. 
If we divide the gaps $A_{n+1}-A_n$ between anti-recurrence numbers into blocks of
three, we now find that there are more blocks: $786$, $678$, $966$, $696$, $669$.
Again, the sum of all blocks is the same and the subsequence $A_{3n+1}$ is an arithmetic
progression.

Kimberling's conjecture for this anti-recurrence is
\begin{equation}\label{eq:2a}
\begin{aligned}
0 &\leq A_n - 7n + 4 &\leq 3, \\
0 &\leq 3B_n^1 - 7n + 6 &\leq 4, \\
0 &\leq 3B_n^2 - 7n + 2&\leq 3.
\end{aligned}
\end{equation}
In this case, the second inequality follows from the first.

\begin{figure}
    \centering
    \includegraphics[width=0.6\linewidth]{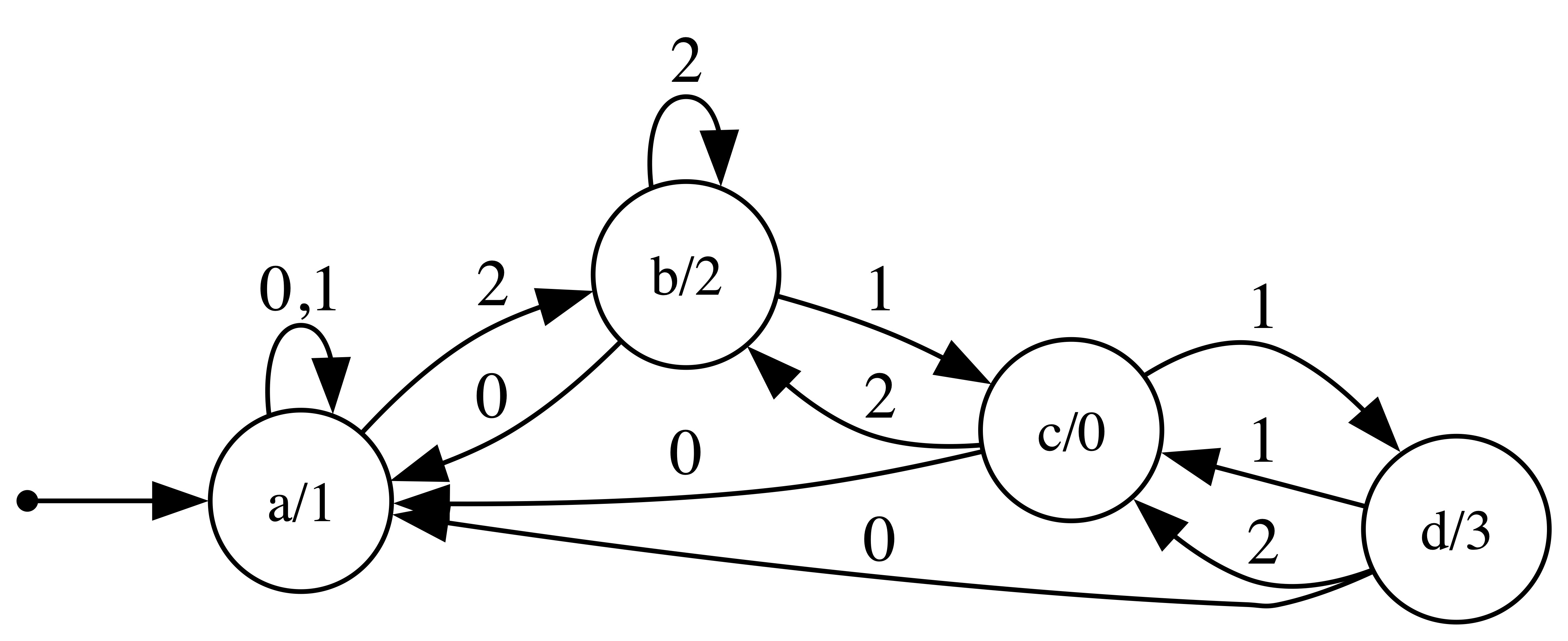}
    \caption{A 3-DFAO for the difference sequence $A_{n+1}-7n-3$
    of the anti-Jacobsthals.
    All inputs $n=0\pmod 3$ end in state $a$ with output $1$, which
    means that $A_{n}=7n-3$ if $n=1\pmod 3$.}
    \label{fig:2a}
\end{figure}
Our guessed automaton \texttt{a21} for the difference sequence of the anti-Jacobsthals is illustrated
in Fig.~\ref{fig:2a}. We use it to implement \seqnum{A304499} in \texttt{Walnut}.
\vspace{-0.5cm}
\begin{flushleft}
\texttt{
\newline
def a304499 "?msd\_3 (n>0) \& s=(7*n-4+a21[n-1])":} 
\end{flushleft}
We check that these numbers are not divisible by three.
\vspace{-0.5cm}
\begin{flushleft}
\texttt{
\newline
eval test "?msd\_3 As (En \$a304499(n,s)) => (Et (s=3*t+1 | s=3*t+2))":} 
\end{flushleft}
which is \texttt{TRUE}. 
We can define the non-anti-recurrence sequences from $A_n$.
The numbers $B_n^1$ are the rounded down $A_n/3$ and $B_n^2=A_n-2B_n^1$. 
\vspace{-0.5cm}
\begin{flushleft}
\texttt{
\newline
def a304497 "?msd\_3 Er \$a304499(n,r) => s=r/3": 
\newline
def a304498 "?msd\_3 Eq,r (\$a304497(n,q) \& \$a304499(n,r)) => s=r-2*q": 
}
\end{flushleft}
We have defined our candidate sequences in \texttt{Walnut}. We still need to
satisfy that our DFAO does indeed produce the right numbers.

\begin{theorem}
    The anti-Jacobsthal numbers satisfy Kimberling's bounds in Equation~\ref{eq:2a}.
\end{theorem}
\begin{proof}
We verify, in exactly the same way as for the anti-Pell numbers, 
that these sequences satisfy the criterion
of Lemma~\ref{lem:1}, starting with the initial conditions
\vspace{-0.5cm}
\begin{flushleft}
\texttt{
\newline
eval test "?msd\_3 \$a304497(1,1) \& \$a304498(1,2) \& \$a304499(1,4)": 
}
\end{flushleft}
which is \texttt{TRUE}. 

We test the \texttt{mex} conditions for the non-anti-recurrence sequences.
\vspace{-0.5cm}
\begin{flushleft}
\texttt{
\newline
eval testB1 "?msd\_3 An,s,t (\$a304497(n,t) \& s>0 \& s<t) => \newline (Em (m<n) \& (\$a304497(m,s)|\$a304498(m,s)|\$a304499(m,s)))": 
\newline
eval testB2 "?msd\_3 An,s (s>1\ \&\ \$a304498(n,s)) => (\$a304497(n,s-1)|(E m (m<n) \& \$a304499(m,s-1)))": 
}
\end{flushleft}
and we test that the additive relation $A_n=2B_n^1+B_n^2$ holds
\vspace{-0.5cm}
\begin{flushleft}
\texttt{
\newline
eval testA "?msd\_3 An,s,t (\$a304497(n,s)\&\$a304498(n,t))=>\$a304499(n,2*s+t)":
}
\end{flushleft}
These are all \texttt{TRUE} and therefore the anti-Jacobsthals satisfy Conjecture~\ref{conj:1}.
To verify Kimberling's bounds in Equation~\ref{eq:2a}, we only
need to verify the first and third inequality.
\vspace{-0.5cm}
\begin{flushleft}
\texttt{
\newline
eval testA \hspace{0.2cm}"?msd\_3 An,s \$a304499(n,s) => (7*n <= s+4 \& s+4 <= 7*n+3)":
}
\texttt{
\newline
eval testB2 "?msd\_3 An,s \$a304498(n,s) => (7*n <= 3*s+2 \& 3*s+2 <= 7*n+3)":
}
\end{flushleft}
which is \texttt{TRUE}.
\end{proof}
This settles Kimberling's conjectures on \seqnum{304499} and \seqnum{304502}. 

\section{The anti-bonacci numbers}

The recurrence relation for $\mathbf{a}=(1,1,\ldots,1)$ given by
\[
X_{n}=X_{n-1}+\cdots+X_{n-k}
\]
starting from the initial conditions $X_n=0$ for $n\leq 0$ and $X_1=1$ 
produces the $k$-bonacci numbers. Apparently, they were first
introduced in~\cite{kbonacci} under the name of $k$-generalized
Fibonacci numbers.
The most familiar cases are the Tribonacci numbers for $k=3$ and
the Tetrabonacci numbers for $k=4$.
Their anti-recurrent counterparts are the anti-Tribonacci sequence
\seqnum{A265389} 
\[
6, 16, 27, 36, 46, 57, 66, 75, 87, 96, 106, 117, 126, 136, 147, 156, \ldots,
\]
and
the anti-Tetrabonacci sequence \seqnum{A299409}
\[
10, 26, 45, 62, 78, 94, 114, 130, 146, 162, 180, 198, 214, 230, 248, \ldots.
\]
Kimberling conjectured bounds on these two sequences that were verified 
by Bosma et al. in~\cite{walnutoeis}
by using \texttt{Walnut}. 
In particular, Bosma et al. showed that
the anti-$k$-bonacci sequence is a sum of a linear sequence and 
a $k$-automatic sequence for $k=3$ and $k=4$. However,
the automata for the difference sequences in~\cite{walnutoeis}
are not that easy to interpret. The automaton for $k=4$
has 10 states, for instance, and that is because
these automata were given in lsd format. 
If we reverse them to msd format, we get much cleaner
machines as shown in Fig.~\ref{fig:A111}
and~\ref{fig:A1111}. Both these DFAO's, as well as the DFAO for
the anti-Fibonacci that we saw earlier, satisfy the following properties:
\begin{itemize}
    \item The number of states is equal to $k$.
    \item The outputs are unique.
    \item All transitions on input $0$ lead back to the initial state.
\end{itemize}
The third property is equivalent to the fact that the subsequence
$A_{kn+1}$ forms
an arithmetic progression with increment $k\kappa=k^3+k$. 

\begin{figure}[h]
    \centering
    \begin{subfigure}[b]{0.4\textwidth}
        \centering
        \includegraphics[width=\textwidth]{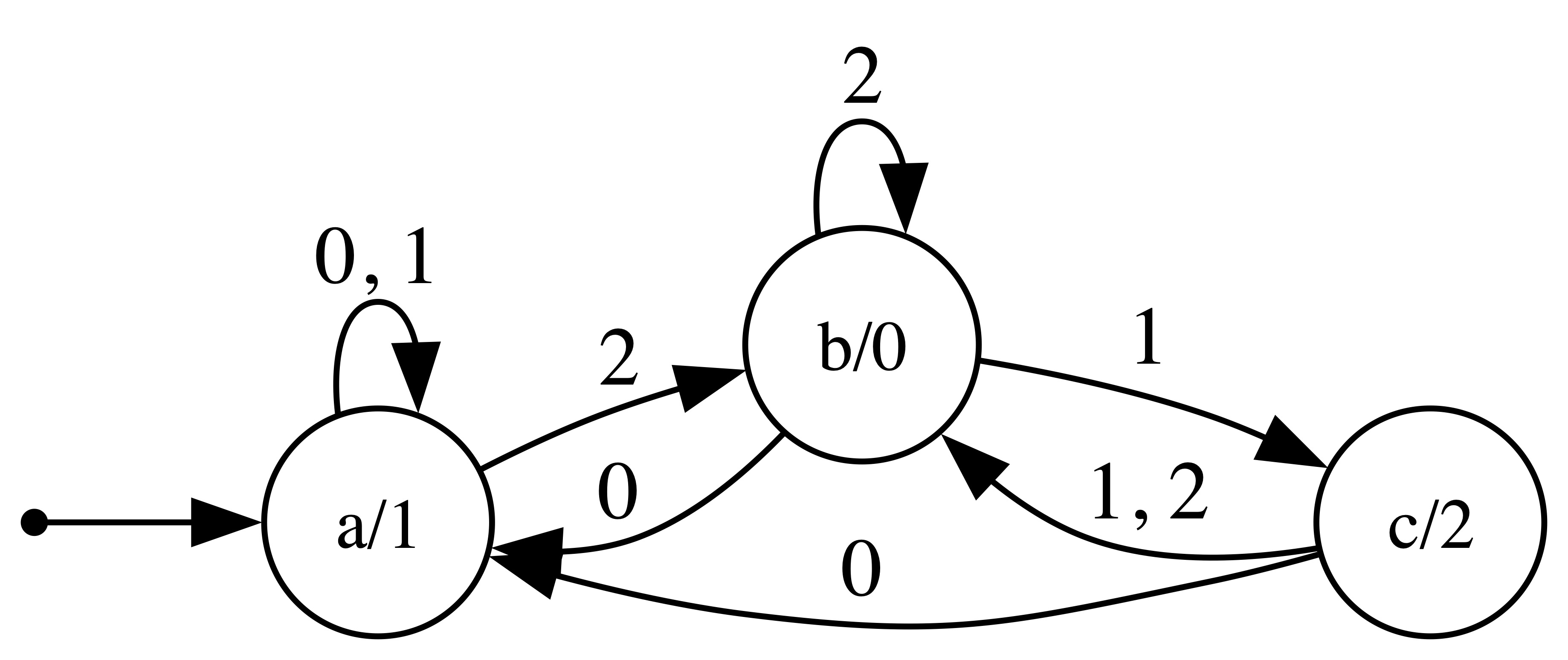}
        \caption{The DFAO for $\mathbf{a}(1,1,1)$}
        \label{fig:A111}
    \end{subfigure}
    \hfill
    \begin{subfigure}[b]{0.5\textwidth}
        \centering
        \includegraphics[width=\textwidth]{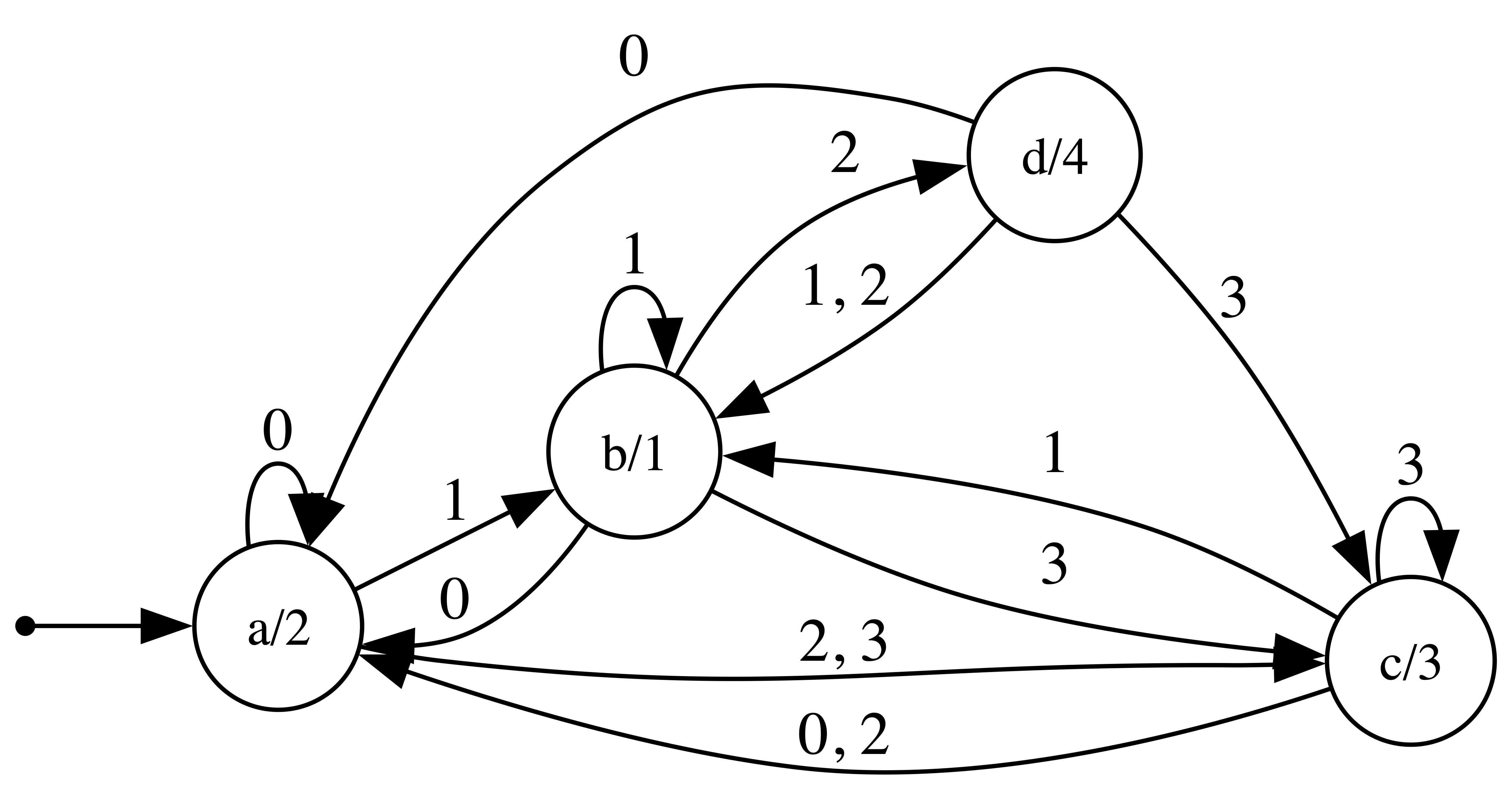}
        \caption{The DFAO for $\mathbf{a}(1,1,1,1)$}
        \label{fig:A1111}
    \end{subfigure}
    \caption{The 3-DFAO for the difference sequence $A_{n+1}-10n-5$ for the anti-Tribonacci sequence
    in $(a)$ and $A_{n+1}-17n-8$ for the anti-Tetrabonacci sequence in $(b)$.
    For instance, in figure $(a)$, $A_{15}=147$ and the input for the automaton is $112$, with output 2. 
    The numbers with output $1$ for the DFAO in $(a)$ correspond to the positions of $0$ in Stewart's choral sequence~\seqnum{A116178}. The automaton in $(a)$ was
    conjectured by Kimberling and Moses~\cite{kimberlingmoses}.
    The automaton in $(b)$ corresponds to the substitution $a\mapsto 21\bar a3$ where
    $\bar a=5-a$, by Cobham's little theorem.
    }
    \label{fig:4}
\end{figure}

The $5$-bonacci numbers, or Pentanacci numbers, are \seqnum{A145029}, but the anti-$5$-bonacci numbers have not yet been entered in the OEIS. The initial numbers are
\[
  15 ,   40 ,   66 ,   95 ,  120 ,  145 ,  170 ,  197 ,  225 ,  250 ,  275 ,  300 ,  327 ,  355 ,  380, \ldots.
\]
We have guessed the automaton \texttt{a11111} for the anti-$5$-bonacci sequence
\[
A_n-26n+13
\]
as illustrated in Fig.~\ref{fig:A11111}. It is possible to check with \texttt{Walnut} that this DFAO does indeed produce the
anti-bonacci sequence for $k=5$. The properties
of the DFAO's that we observed for $k=2,3$ and $4$ are again satisfied. We will now prove
that these hold for all $k$-antibonaccis.
\begin{figure}[h]
    \centering
    \includegraphics[width=0.5\linewidth]{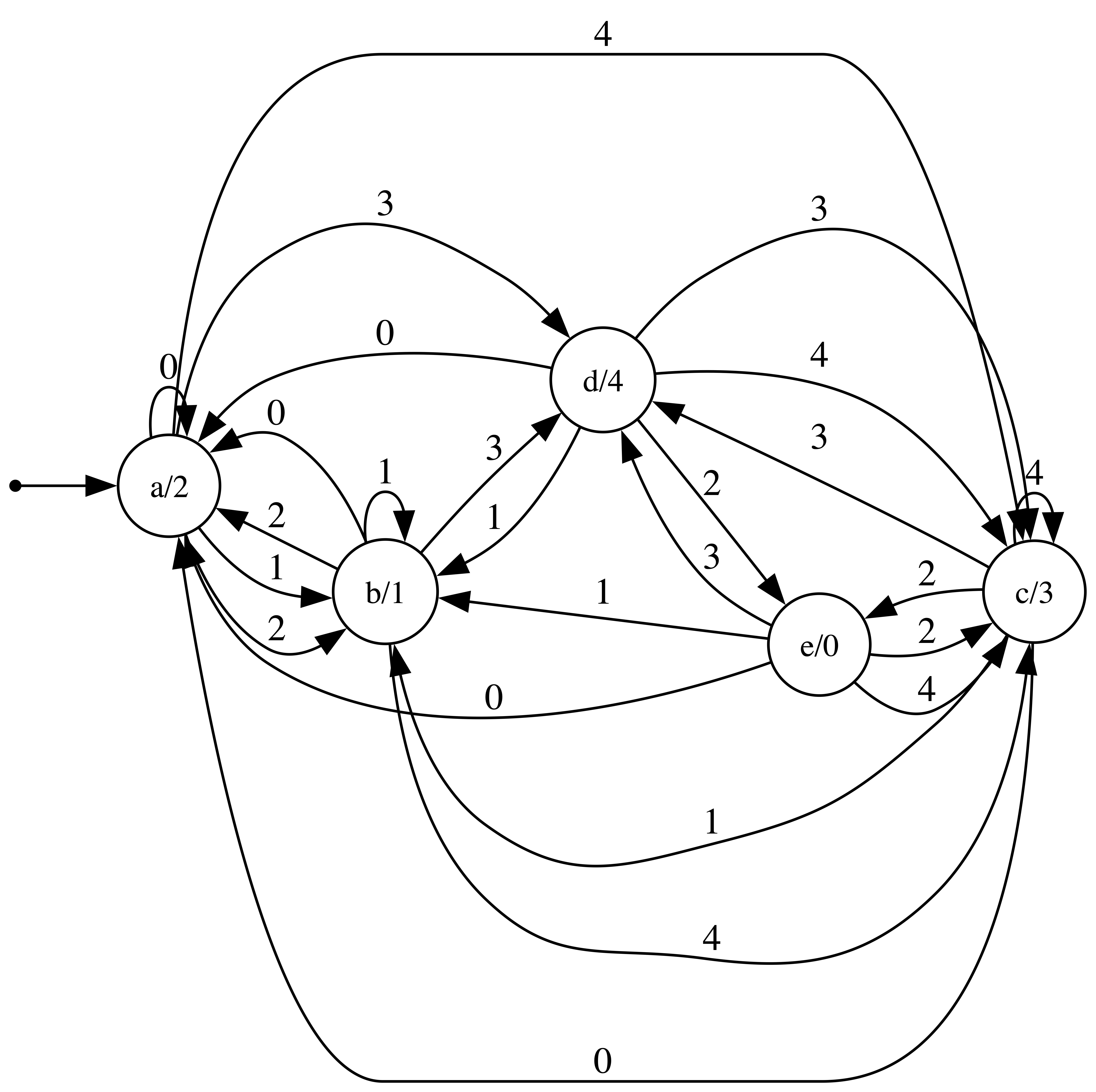}
    \caption{The 5-DFAO for the difference sequence 
    $A_{n+1}-26n-13$ of the anti-$5$-bonacci numbers. If we identify the states $a$ to $e$
    with their outputs $0$ to $4$, then by Cobham's little theorem this DFAO corresponds
    to the substitution $x\mapsto 21y43$ for $y=3-x\hspace{-0.2cm}\mod 5$ except for $x=4$
    and for which $4\mapsto 21033$.}
    \label{fig:A11111}
\end{figure}

\medbreak
We fix $k$ and denote the anti-$k$-bonacci numbers by $A_n$ without
including $k$ in the notation. The difference sequence is $A_n-\kappa n$
with $\kappa=k^2+1$.
$A_1$ is equal to the triangular number 
$t_k=1+2+\cdots+k=\binom{k+1}2$.
We consider consecutive intervals of length $k^2+1$
\[
I_n=[(n-1)\kappa+1,n\kappa].
\]
Recall that the set $\{B_m^1,\ldots,B_m^k\}$ is the \emph{$B$-block} that generates $A_m$.
We will show that each $I_n$ contains one anti-bonacci $A_{n}$ and $k$ such $B$-blocks. We can thus
associate $A_{n}$ to $I_n$, which generates $k$ anti-bonaccis $A_{k(n-1)+j}$ for $j=1,\ldots,k$.
That gives a substitutive rule for the anti-bonaccis. Since $I_n$ contains only one anti-recurrence,
at most one of the blocks is not an interval. Lemma~\ref{lem:2} implies that at least $k-2$ of
the $A_{k(n-1)+j}$ are $k^2$ apart. Modulo $\kappa$, the next anti-recurrence decreases by $1$
if this is the case. The first $k$ numbers of $I_n$ form the first $B$-block, which explains
why the $A_{kn+1}$ form an arithmetic progression. The following lemma makes this precise.

\begin{lemma}\label{lem:an}
    Let $i=\lfloor\frac {k}2\rfloor$ and $\kappa=k^2+1$ for $k>2$. 
    Then $A_n\in I_n$ and \[A_n=ik+a_n\hspace{-0.2cm}\mod \kappa\]
    for some $1\leq a_n\leq k$ if $k$ is even and $i+1\leq a_n\leq i+k$
    if $k$ is odd.
\end{lemma}
\begin{proof}
    We have $A_1=t_k\in [1,\kappa]=I_1$.
    If $k$ is odd, then $t_k=ik+k$ and if it
    is even, then $t_k=ik+\frac  k2$. So $a_1=k$ if $k$ is odd
    and $a_1=\frac k2$ if it is even. The initial interval
    $I_1$ contains $k$ $B$-blocks, all of which are intervals except for one.
    The possible exception is either the $i+1$-st block or the $i+2$-nd block. 
    That is our inductive hypothesis. 

    A $B$-block determines an anti-bonacci, and therefore each $I_j$
    determines $k$ anti-bonaccis.  By Lemma~\ref{lem:2},
    if blocks are consecutive intervals,
    then they determine anti-bonaccis that are $k^2=\kappa-1$ apart. Modulo $\kappa$,
    the next anti-bonacci decreases by $1$.
    If they are not consecutive intervals, then the anti-bonaccis
    are further apart.
    Suppose that $A_j=ik+a_j$ for $1\leq a_j\leq k$.
    If $a_j=1$, then the $i+1$-st block in $I_j$ is an interval
    that generates an anti-bonacci $A$ that
    is equal to $k^2+k$ plus the previous anti-bonacci.
    If $1<a_j\leq k$ then the $i+1$-st block in $I_j$ is not an interval.
    It generates an anti-bonacci $A$ that is $k^2+k+1-a_j$
    plus the previous anti-bonacci. The next anti-bonacci is $k^2+a_j-1+A$.
    If $1\leq a_j\leq k$ then
    the interval $I_j$ generates $k$ anti-bonaccis with first differences
    \begin{equation}\label{eq:block}
    \overbrace{k^2 \cdots k^2}^{i-1}
    \ x(2k^2 + k - x) \
    \overbrace{k^2 \cdots k^2}^{k - i - 1}
    \end{equation}
    with $x=k^2+k+1-a_j$.
    For instance, if $k=4$ then $i=2$ and the differences are
    $16,x,36-x,16$ for $x$ in between 17 and 20. 
    The final difference is $k^2$, since the final block of $I_j$
    and the first block of $I_{j+1}$ are consecutive intervals.
    Note that the total sum of the first differences is $k\kappa$
    and therefore $A_{jk+1}=A_{(j-1)k+1}\hspace{-0.2cm}\mod \kappa$.

    From these first differences we can compute $a_{(j-1)k+1},\ldots, a_{jk}$
    for the anti-bonaccis that are generated from $I_j$. 
    In Equation~\ref{eq:block} the initial $i-1$ differences 
    are $-1$ mod $\kappa$. These are followed by $k+1-a_j$ and $a_j-2$ mod $\kappa$, followed by
    $k-i-1$ differences $-1$. If $1\leq a_j\leq k$ then
    \[
    a_{(j-1)k+\ell} = 
    \left\{
        \begin{array}{lll}
            a_1 + 1 - \ell & \text{if } 1 \leq \ell \leq i, \\
            a_1 + 1 - i + k - a_j & \text{if } \ell = i+1, \\
            a_1 + 1 + k - \ell & \text{if } i + 2 \leq \ell \leq k.
        \end{array}
    \right.
    \]
The $i+1$-st entry is the only one that depends on $a_j$. 
If $k$ is even, we have
$i=\frac k2=a_1$, and so
    \begin{equation}\label{eq:k-even}
    a_{(j-1)k+\ell} = 
    \left\{
        \begin{array}{ll}
            i + 1 - \ell & \text{if } 1 \leq \ell \leq i, \\
            k + 1  - a_j & \text{if } \ell = i+1, \\
            k + 1 - (\ell-i) & \text{if } i + 2 \leq  \ell \leq k.
        \end{array}
    \right.
    \end{equation}
These numbers are between $1$ and $k$. Therefore, if $k$ is even then our inductive
hypothesis implies that the first $kn$ values of $a_j$ are in between $1$ and $k$.
This settles the result for $k$ even. 

For $k$ is even, we have $a_1=\frac k2$.
For $k$ is odd, we have a larger value $a_1=k$. The $A$'s can be in the $i+1$-st
and $i+2$-nd block for odd $k$. In particular, this is the case if $a_j>k$. The first differences then are
    \begin{equation}\label{eq:block}
    \overbrace{k^2 \cdots k^2}^{i}
    \ x(2k^2 + k - x) \
    \overbrace{k^2 \cdots k^2}^{k - i - 2}
    \end{equation}
with $x=k^2+2k+1-a_j$. As before, we can compute the $a$'s from these differences.
    \[
    a_{(j-1)k+\ell} = 
    \left\{
        \begin{array}{ll}
            a_1 + 1 - \ell & \text{if } 1 \leq \ell \leq i+1, \\
            a_1 - i + 2k - a_j & \text{if } \ell = i+2, \\
            a_1 + k + 1 - \ell & \text{if } i + 3 \leq \ell \leq k.
        \end{array}
    \right.
    \]
If $k$ is odd, then $a_1=k$ and $i=\frac{k-1}2$. If $a_j\leq k$ then we get
    \begin{equation}\label{eq:k-odd1}
    a_{(j-1)k+\ell} = 
    \left\{
        \begin{array}{lll}
            k + 1 - \ell & \text{if } 1 \leq \ell \leq i, \\
            k + i+ 2 - a_j & \text{if } \ell = i+1, \\
            2k + 1 - \ell & \text{if } i + 2 \leq \ell \leq k.
        \end{array}
    \right.
    \end{equation}
These numbers are in $[i+2,i+k]$. If $a_j>k$ then we get
    \begin{equation}\label{eq:k-odd2}
    a_{(j-1)k+\ell} = 
    \left\{
        \begin{array}{ll}
            k + 1 - \ell & \text{if } 1 \leq \ell \leq i+1, \\
            2k + i + 1 - a_j & \text{if } \ell = i+2, \\
            2k + 1 - \ell & \text{if } i + 3 \leq \ell \leq k.
        \end{array}
    \right.
    \end{equation}
These numbers are in $[i+1,i+k]$.
\end{proof}

Equations~\ref{eq:k-even} and~\ref{eq:k-odd1}, ~\ref{eq:k-odd2} for the remainders $a_n$
allow us to generalize Zaslavsky's formula from $k=2$ to $k>2$. The following result
settles conjecture~\ref{conj:1} for anti-bonaccis.

\begin{theorem}\label{thm:abonacci}
    Let $i=\lfloor\frac {k}2\rfloor$.
    There exists a $k$-uniform substitution $\sigma$ on $\{1,2,\ldots,k\}$ if $k$ is even
    and $\{0,1,\ldots,k-1\}$ if $k$ is odd
    with unique fixed
    point $\omega=(i_n)$ such that
    \[
    A_n=\kappa(n-1)+t_k-i+i_{n-1}.
    \]    
    All $\sigma(j)=w_j$ have initial digit $i$ and therefore $\lim_{n\to\infty}\sigma^n(i)=\omega$. 
\end{theorem}
\begin{proof}
    By Lemma~\ref{lem:an}, $A_n=\kappa(n-1)+ik+a_n$.
    If $k$ is even, then $ik=t_k-i$, and so we get
    $A_n=\kappa(n-1)+t_k-i+a_n$ for numbers $a_n\in[1,k]$. Equation~\ref{eq:k-even} describes a $k$-substitution
    $a_j\mapsto w$ where the digits $w_\ell$ are given by $a_{(j-1)k+\ell}$, which is
    independent on $a_j$, except if $\ell=i+1$, when the digit is $k+1-a_j$.
    The initial digit of each substitution word is equal to $i$.
    We write $i_{n-1}=a_n$ to comply with the rule that automatic sequences
    start with index $0$. 
    
    If $k$ is odd, then $ik=t_k-k$, and so we get that  $A_n=\kappa(n-1)+t_k-k+a_n$
    for numbers $a_n\in [i+1,i+k]$. 
    If we write $i_{n-1}=a_n-k+i$ then we get that $A_n=\kappa(n-1)+t_k+i+i_{n-1}$.
    Since $k=2i+1$ if $k$ is odd, $i_{n-1}\in[0,k-1]$.
    Equations~\ref{eq:k-odd1} and~\ref{eq:k-odd2} 
    describe a $k$-substitution $a_j\mapsto w_j$ with initial digit $a_1$. 
    Then we get that $A_n=\kappa(n-1)+t_k-i+i_{n-1}$ and the initial
    digit of the substitution words is~$i$. This confirms our observations on
    the DFAO's for the anti-bonaccis.
\end{proof}

\section{Rusty numbers and other anti-recurrences}

The recurrence
\[
X_{n+1}=dX_n+X_{n-1}
\]
produces the so-called metallic or metallonacci numbers~\cite{metallic2, metallic}. 
It is impossible to resist the temptation to say that the $A_n$
for the linear form $\textbf{a}={(1,d)}$ are the \emph{rusty numbers}.
The $3$-rusty numbers are
\[
7,    15,    23,    35,    43,    51,    62,    71,    79,    87,    99,   107,   115,   123,   131, 142, 151, \ldots
\]
We can guess a 4-DFAO for its difference sequence, which is illustrated in Fig.~~\ref{fig:a13}. The \texttt{Walnut} verification for the anti-Pell and anti-Jacobsthal sequences can also be applied to this DFAO, to check that it indeed
produces the difference sequence. Note that all $0$-transitions lead back to the initial state, which implies that the subsequence $A_{4n+1}$ is an arithmetic progression. However, this does not apply to the $4$-rusty numbers. 
\[
    9,    19,    29,    39,    54,    64,    74,    84,    98,   109,   119,   129,   139,   154,   164, 174, 184,\ldots
\]
The subsequence $A_{5n+1}$ is equal to the arithmetic progression $A_1+55 n$ up until $n=348$, when $A_{1741}\not=9+55\cdot 348$. The metallic numbers are well-studied and share many of the properties of the
Fibonacci numbers. Suprisingly, it is a challenge to prove, or disprove, the conjecture for the rusty numbers.

The general quadratic recurrence $X_{n+1}=pX_n+qX_{n-1}$ with arbitrary initial values produces the
Horadam numbers~\cite{horadam}. We will show that the anti-Horadam numbers have an automatic difference
sequence if $p\leq 2$.
\begin{figure}
    \centering
    \includegraphics[width=0.7\linewidth]{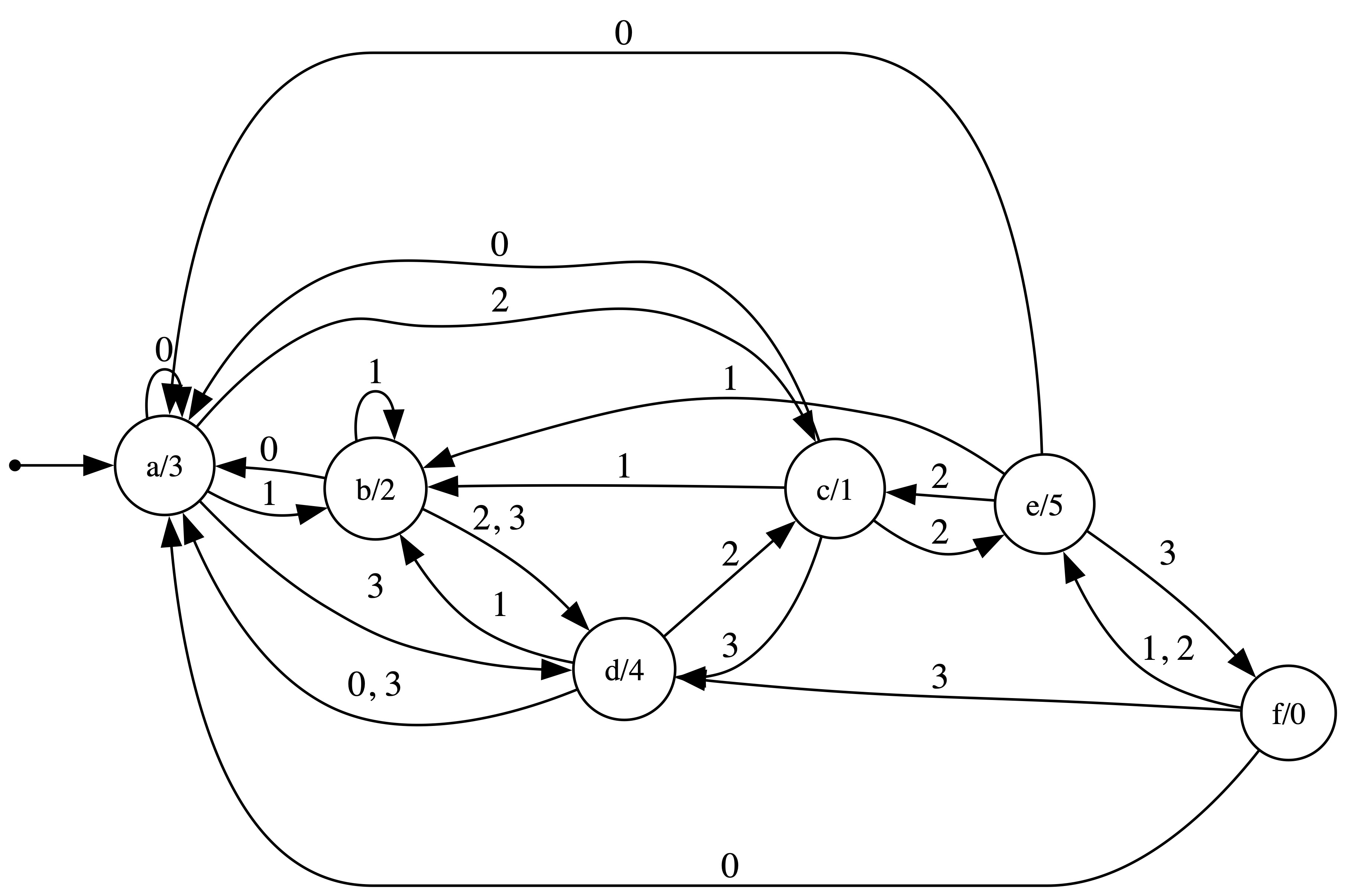}
    \caption{A $4$-DFAO for the difference sequence $A_n-9n-4$ for
    the linear form $\mathbf{a}=(1,3)$.}
    \label{fig:a13}
\end{figure}

\begin{definition}
    A positive linear form $\mathbf a$ of dimension $k$ and trace $\tau$
    is $A_1$-bounded if \[A_1\leq (k-1)\tau+2,\]
    where $A_1$ is the first anti-recurrence number defined by $\mathbf a$.
\end{definition}
The linear form $\mathbf a=(a_1,a_2)$ of the anti-Horadam numbers is $A_1$-bounded
if $a_2\leq 2$.

\begin{lemma}\label{lem:ineq}
    $\tau+t_{k-1}\leq A_1$ and the inequality is strict if $A_1$ is not
    anti-Fibonacci.
\end{lemma}
\begin{proof}
    The inequality follows from
    $A_1=\sum_{j=1}^k ja_j=\tau+\sum_{j=1}^k (j-1)a_j\geq \tau+t_{k-1}.$
    This inequality is strict if $k>2$ or one of the $a_j>1$.
\end{proof}

\begin{lemma}\label{lem:3}
    Let $\mathbf{a}$ be an $A_1$-bounded linear form.
    Then $A_n\in I_n=[\kappa(n-1)+1,\kappa n]$, with
    $\kappa=k\tau+1$. The initial  
    $B$-block of each $I_n$ is an interval.
\end{lemma}

\begin{proof}
     By induction. For $A_1$ we need to prove that
     $k\leq A_1\leq \kappa$.
     The left-hand inequality follows from $A_1\geq\tau+t_{k-1}\geq k+t_{k-1}$. The right-hand
     inequality is immediate.
     
     By our inductive assumption, each $I_j$ contains one anti-recurrence and $\tau$ blocks.
     Therefore, each $I_j$ generates $\tau$ anti-recurrence numbers, and from our inductive
     hypothesis we can generate $n\tau$ recurrence numbers. We only need to check $A_{n+1}$.
     The initial block of $I_j$ is an interval which generates $A_{(j-1)\tau+1}=A_1+(j-1)\kappa$. 
     This gives the familiar arithmetic progression. 

     If $B$-blocks are consecutive intervals, then they generate $A_h$ and $A_{h+1}$ such that
     $A_{h+1}-A_h=\kappa-1$ by Lemma~\ref{lem:2}. 
     Modulo $\kappa$, the next number $A_{h+1}$ reduces by one. 
     Since there is only one anti-recurrence
     number, at least $\tau-1$ of the blocks are intervals, and at least $\tau-2$ of these are 
     consecutive to a preceding block that is an interval. 
     There is one $B$-block that is either not an interval, or not consecutive to a preceding block.
     The latter happens if the anti-recurrence number is between two $B$-blocks. In that case,
     $A_{h+1}-A_h=(k+1)\tau=\kappa+\tau-1$. 
     There are at most $\tau - 1$ reductions by one for the $\tau$ anti-recurrence numbers that
     are generated by $I_j$. There are at most $2$ increases. 
     The interval $I_j$ generates $\tau$ anti-recurrences, which have $\tau-1$ differences. If
     we include the first anti-recurrence of $I_{j+1}$, then we get $\tau$ differences. The total
     sum of these differences is zero, since $A_{j\tau+1}=A_1\hspace{-0.2cm}\mod\kappa$. It
     follows that each anti-recurrence $A_h$ that is generated by $I_j$ is in the range
     $[A_1-\tau+1,A_1+\tau-1]$ modulo $\kappa.$ By Lemma~\ref{lem:ineq} we have that
     $A_1-\tau+1\geq t_{k-1}+1\geq k$. By $A_1$-boundedness we have that $A_1+\tau-1\leq k\tau+1=\kappa$.
     The numbers $A_h$ that are generated by $I_j$ are contained in $I_h$ and are not in its
     initial interval of length $k$. In particular, $A_{n+1}$ meets the required conditions.
\end{proof}

It follows from the proof of this lemma that the subsequence $A_{n\tau+1}$ is an arithmetic
progression if $\mathbf a$ is $A_1$-bounded.

\begin{theorem}\label{main}
    If $\mathbf a$ is $A_1$-bounded, then it generates an anti-recurrence sequence $A_n$
    such that $A_n-\kappa n$ is $\tau$-automatic.
\end{theorem}

\begin{proof}
    In the proof of Lemma~\ref{lem:3} we saw that $I_j$ generates $\tau$ anti-recurrence
    numbers. This process depends only on the value of $A_j\hspace{-0.2cm}\mod\kappa$.
    Furthermore, we found that the anti-recurrence numbers are all in $[A_1-\tau+1,A_1+\tau-1]$
    if we compute modulo~$\kappa$. Therefore, there are only $2\tau-1$ possible values.
    We have a uniform substitution of length $\tau$ on an alphabet of size $2\tau-1$.
    By Cobham's little theorem, the difference sequence $A_n-\kappa n$ is $\tau$-automatic.
    It can be recognized by a DFAO with at most $2\tau-1$ states.
\end{proof}

\section{Final remark}

We have shown that a specific class of anti-recurrence sequences are the sums of a linear
sequence and an automatic sequence. We left much to be explored, most notably the extension
of 
{Theorem \ref{main}} to general anti-recurrence sequences. 
Does the conjecture hold without the restriction of $A_1$-boundedness?
Are the rusty numbers sums of linear sequences and automatatic sequences?
There is the more general class of complementary sequences 
that goes back to Fraenkel~\cite{fraenkel}. 
Is it possible to single out complementary sequences
that are sums of linear sequences and automatic sequences? 

\section{Acknowledgement}

Gandhar Joshi received support from the London Mathematical Society via grant SC7-2425-18.

\bibliographystyle{siam}
\bibliography{antirecurrence}

\bigskip
\hrule
\bigskip
\noindent 2020 {\it Mathematics Subject Classification}:
Primary  11B85; Secondary 11B37, 68R15.

\noindent \emph{Keywords: } linear anti-recurrence, automatic sequence, $k$-bonacci numbers. 

\bigskip
\hrule
\bigskip

\noindent (Concerned with sequences
\seqnum{A075326},
\seqnum{A096268}, \seqnum{A116178},
\seqnum{A249032},
\seqnum{A265389}, \seqnum{A299409}, \seqnum{A304499}, \seqnum{A304502}.
)

\end{document}